\documentclass[12pt]{amsart}
\usepackage{amscd}
\usepackage{amssymb}
\usepackage{hyperref}
\usepackage{graphicx}
\usepackage{nicematrix}
\usepackage{lipsum,amsmath}

\newcommand{\ZZ}{\mathbb Z}
\newcommand{\NN}{\mathbb N}

\newcommand{\C}{\mathcal C}
\newcommand{\QQ}{\mathbb Q}

\newtheorem{lemma}{Lemma}[section]
\newtheorem{corollary}[lemma]{Corollary}
\newtheorem{theorem}[lemma]{Theorem}
\newtheorem{proposition}[lemma]{Proposition}
\newtheorem{definition}[lemma]{Definition}
\newtheorem{remark}[lemma]{Remark}
\newtheorem{example}[lemma]{Example}
\newtheorem{question}[lemma]{Question}
\newtheorem{notation}[lemma]{Notation}

\advance\headheight1.15pt

\begin{document}

\title[Circuit binomials of weighted oriented graphs]{On circuit binomials of toric ideals of weighted oriented graphs}

\author[Ramakrishna Nanduri]{Ramakrishna Nanduri$^*$}
\address{Department of Mathematics, Indian Institute of Technology
Kharagpur, West Bengal, INDIA - 721302.}
\email{nanduri@maths.iitkgp.ac.in}
\author[Tapas Kumar Roy]{Tapas Kumar Roy$^{\dag}$}
\address{Department of Mathematics, Indian Institute of Technology
Kharagpur, West Bengal, INDIA - 721302.}
\email{tapasroy147@kgpian.iitkgp.ac.in}
\thanks{$^\dag$ Supported by PMRF research fellowship, India.}
\thanks{$^*$ Supported by SERB grant No: CRG/2021/000465, India}
\thanks{{\bf AMS Classification 2020:} 13F65, 13A70, 05E40, 05C50, 05C38. }

\begin{abstract} 
In this work, we classify the circuit binomials of any weighted oriented graph $D$ and we explicitly compute the circuit binomials of $D$ in terms of the minors of the incidence matrix of $D$. We show that the circuit binomials of any weighted oriented graph $D$ are the primitive binomials corresponding to one of the classes: (i) a balanced cycle, (ii) two unbalanced cycles sharing a vertex, (iii) two unbalanced cycles connected by a path, (iv) two unbalanced cycles sharing a path. We explicitly prove a formula for the primitive binomial generator of the toric ideal $I_D$ in terms of the minors of the incidence matrix of $D$, where $D$ is as in (i), (ii), (iii) and (iv). Thus we explicitly compute all the circuit binomials $\C_D$ of any weighted oriented graph $D$. If $D$ is a weighted oriented graph which has at most two unbalanced cycles such that no two balanced cycles share a path in $D$ and no balanced cycle in $D$ shares an edge with the path which connects the two unbalanced cycles in $D$ if it exists, then we show that $I_D$ is a strongly robust circuit ideal and it has complete intersection initial ideal. For this class of ideals, we explicitly compute the Betti numbers.  
\end{abstract}
\maketitle

\section{Introduction} \label{sec1}

Toric ideals are important in modern theory of algebra because of their variety of applications in diverse research fields such as Commutative Algebra, Combinatorics, Algebraic Geometry, Integer Programming, Semigroup Rings, Combinatorial Optimization, Coding Theory and Algebraic  Statistics etc, see \cite{ds, v01, chks06, cdss, cls} and the references therein. These ideals are a special type of ideals in a polynomial ring known as binomial ideals. 

Let $R=K[x_1,\ldots,x_n]$, where $K$ is a field. Let $A=[{\bf a_1} \ldots {\bf a_m}]$ be an $ n\times m$ matrix of non-negative integers with columns ${\bf a_1}, \ldots, {\bf a_m}$. Let $S=K[e_1, \ldots ,e_m]$ be the polynomial ring in the variables $e_1,\ldots,e_m$. Then define a $K$-algebra homomorphism $\phi$ : $K[e_1, \ldots, e_m]\rightarrow K[x_1, \ldots , x_n]$, as $\phi(e_i) =\textbf{x}^{\textbf{a}_{i}}$. Then the kernel of $\phi$ is called the {\it toric} ideal of $A$ (or the toric ideal of the monomial ideal $({\bf x}^{\bf a_1},\ldots,{\bf x}^{\bf a_m})$), and we denote by $I_{A}$. Note that $I_A$ is a binomial ideal in $S$. A binomial ${\bf e^u}-{\bf e^v}$ in $I_A$ is called a {\it primitive} binomial if there exists no other binomial ${\bf e^{u^{\prime}}}-{\bf e^{v^{\prime}}}$ in $I_A$ such that ${\bf e^{u^{\prime}}}\vert {\bf e^u}$ and ${\bf e^{v^{\prime}}} \vert {\bf e^v}$. The set of primitive binomials in $I_A$ is called the {\it Graver basis} of $I_A$ and denoted by $Gr_A$. Note that $I_A$ is generated by $Gr_A$. A binomial ${\bf e^u}-{\bf e^v}$ in $I_A$ is called a {\it circuit} of $I_A$ if it has minimal support with respect to set inclusion. The set of circuits of $I_A$ is denoted by $\mathcal{C}_A$. We call a toric ideal $I_A$ is a circuit ideal of it is generated by $\C_A$.  It is well known that $\mathcal{C}_A \subseteq  Gr_A$. 

Toric ideals of monomial ideals are of general interest of many researchers. Toric ideals of square-free monomial ideals are studied by using the notion of hypergraphs, see \cite{v01, ps14, nn22}. In particular, the toric ideals of quadratic square-free monomial ideals (i.e, edge ideals of finite simple graphs) are well studied. For a simple graph $G$, the toric ideal of the edge ideal of $G$, is generated by the binomials corresponding to primitive even closed walks in $G$, see \cite{oh99, v17} and these binomials are square-free. It is a very hard problem to find the generators of toric ideals of non square-free monomial ideals. In the literature, very few results are known for the non square-free monomial ideals case. Recently in \cite{bklo}, Biermann, Kara, Lin and O'Keefe studied the toric ideals of non-square-free monomial ideals arise out of weighted oriented graphs and they characterized principal toric ideals of weighted oriented graphs. However, not much is known about the binomial generators, in particular the primitive binomials, of toric ideals of edge ideals of weighted oriented graphs. The toric ideal of the edge ideal of a weighted oriented graph D, we simply called as the toric ideal of the weighted oriented graph D and we denote by $I_D$. The weighted oriented graphs are of independent interest to study because of their applications in various fields, viz, Coding Theory, Combinatorics, Commutative Algebra, Algebraic Geometry etc, see \cite{v01, mpv17}. Unlike the case of simple graphs, the toric ideals of weighted oriented graphs need not be generated by square-free binomials. Therefore the problem is two-fold, to find the supports as well as exponents of the primitive binomials of the toric ideals of weighted oriented graphs. In this work, we study the circuit binomials of any weighted oriented graph $D$, and we explicitly compute the supports and exponents of circuit binomials of $I_D$ in terms of the minors of the incidence matrix of $D$. We show that the circuits of $D$ are precisely the binomials corresponding to four types of graphs: a balanced cycle, or two unbalanced cycles share a vertex, or two unbalanced cycles connected by a path, or two unbalanced cycles share a path (see Theorem \ref{sec3pro1}).

A vertex weighted oriented graph (with no loop edges and no parallel edges), we simply call as weighted oriented graph $D$ is a triplet $(V(D),E(D), {\bf w})$, where $V(D)$ is the vertex set of $D$, $E(D)$ is the edge set of $D$ with a orientation to each edge and ${\bf w}$ is a weight function which assign a weight to each vertex of $D$. Then the toric ideal of $D$ is the toric ideal of the incidence matrix $A(D)$, of $D$ (or the toric ideal of the edge ideal of $D$) and  we denote by $I_{D}$. Let $Gr_{D},\mathcal{C}_{D}$ denote Graver basis of $I_{D}$ and the set of circuit binomials of $I_{D}$ respectively. Let $D$ be a weighted oriented graph and $D$ has at most two unbalanced cycles such that (i) if $D$ has exactly two unbalanced cycles connected by a path $P$, then no other balanced cycle in $D$ shares an edge with the path $P$,  and  (ii) no two balanced cycles in $D$ share a path. Then we show that $Gr_{D}=\mathcal{C}_{D}$ and $I_{D}$ is strongly robust, that is, $Gr_D$ is a minimal generating set of $I_D$ (Theorem \ref{sec3thm1}). We prove that for any weighted oriented graph $D$, the generators of $I_{D}$ are independent of weights of sink vertices, that is, if $D^{\prime}$ is the graph obtained from $D$ by replacing all weights of sink vertices by $1$, then $I_{D}=I_{D^{\prime}}$ (Proposition \ref{sec3pro2}). 
For any $m\times n$ matrix $A$, we denote $M_k(A[i_1,\ldots,i_{m-k}|j_1,\ldots,j_{n-k}])$, the $k^{th}$ minor of $A$ by deleting the rows $i_1,\ldots,i_{m-k}$ and deleting columns $j_1,\ldots,j_{n-k}$ from $A$.
For any balanced cycle $\C_{2n}$, we show that $I_{\C_{2n}}$ is generated by the primitive binomial corresponding to the vector 
$$\frac{1}{d}\Bigg ((-1)^{i+1}M_{2n-1}(A(\C_{2n})[1|i]) \bigg)_{i=1}^{2n} \in \ZZ^{2n},$$ where $d$ is the gcd of all these minors $M_{2n-1}(A(\C_{2n})[1|i])$'s (Theorem \ref{sec4thm1}). Let $D$ be the weighted oriented graph comprised of two unbalanced cycles $\C_m$ and $\C_n$ sharing a single vertex. Then we show that the toric ideal $I_D$ is generated by the primitive binomial corresponding to the vector
$$\displaystyle \frac{1}{d}\Bigg( \left((-1)^{i+1}pM_{m-1}(A({\C_m})[1|i])\right)_{i=1}^{m}, \left((-1)^{i}qM_{n-1}(A({\C_n})[1|i])\right )_{i=1}^{n} \Bigg) \in \ZZ^{m+n},$$ 
where $q=det(A(\C_m)), p=det(A(\C_n))$, and $d$ is the gcd of all the entries (without sign) in the this vector (Theorem \ref{sec4thm2}). Now, let $D$ be the weighted oriented graph comprised of two unbalanced cycles $\C_m$ and $\C_n$ connected by a path $P$ of length $k$. Then we show that $I_D$ is generated by the primitive binomial corresponding to the vector 
$$\frac{1}{d} \Bigg(((-1)^{i+1} p r_{i})_{i=1}^{m},((-1)^{i} pq r_{m+i})_{i=1}^{k},((-1)^{i+k} q r_{m+k+i})_{i=1}^{n} \Bigg) \in \ZZ^{m+k+n}, \; \text{where}$$ 
\begin{eqnarray*}
r_{i} &=& M_{k}(A(P)[k+1|\emptyset])M_{m-1}(A({\C_m})[1|i]),\; \text{ for }\;1\le i\le m, \\ 
r_{m+i}&=& M_{k-1}(A(P)[1,k+1|i]),\; \text{ for }\;1\le i\le k, \\ 
r_{m+k+i}&=& M_{k}(A(P)[1|\emptyset])M_{n-1}(A({\C_n})[1|i]),\; \text{ for }\;1\le i\le n,
\end{eqnarray*}
$p,q$ as above and $d=\mbox{gcd}((|p|r_{i})_{i=1}^{m}, (|pq|r_{m+i})_{i=1}^k, (|q|r_{m+k+i})_{i=1}^n)$ (see Theorem \ref{sec4thm3}). Finally, let $D$ be the weighted oriented graph comprised of two unbalanced cycles $\C_m,\C_n$ sharing a path $P$ of length $k\geq 1$ such that the outer cycle is unbalanced. Then we prove that the toric ideal $I_D$ is generated by the primitive binomial corresponding to the vector 

$$\frac{1}{d}\Bigg (((-1)^{i+1}sr_{i})_{i=1}^{k},((-1)^{i+m-k+1}pr_{k+i})_{i=1}^{m-k},((-1)^{i+m-k}qr_{m+i})_{i=1}^{n-k} \Bigg ) \in \ZZ^{m+n-k},$$
\begin{eqnarray*}
\text{ where, } 
r_{i} &=& M_{k-1}(A(P)[1,k+1|i]),\; \text{ for }\;1\le i\le k, \\ 
r_{k+i}&=& M_{m-k-1}(A({\C_m}\setminus P)[1,m-k+1|i]),\; \text{ for }\;1\le i\le m-k, \\ 
r_{m+i}&=& M_{n-k-1}(A({\C_n}\setminus P)[1,n-k+1|i]),\; \text{ for }\;1\le i\le n-k,
\end{eqnarray*}
$q=\mbox{det}(A({\C_m})),p=\mbox{det}(A({\C_n})),s=\mbox{det}(A({\C}))$, and \\ 
$d=\mbox{gcd}((|s|r_{i})_{i=1}^{k},(|p|r_{k+i})_{i=1}^{m-k}, (|q|r_{m+i})_{i=1}^{n-k})$ (see Theorem \ref{sec4thm4}). Thus for any weighted oriented graph $D$ as one of above, then the supports of $f_{\bf x}^+, f_{\bf x}^-$ are depend on the signs of the determinants of incidence matrices of unbalanced cycles in $D$, where $f_{\bf x}=f_{\bf x}^+-f_{\bf x}^-$ is the primitive binomial generator of $I_D$.   

We organize the paper as follows. In section \ref{sec2}, we recall all the definitions, notations and basic results that are required to prove our main results in the subsequent sections. In section \ref{sec3}, we study the circuit binomials of weighted oriented graphs. Finally, in section \ref{sec4}, we derive an explicit formula for the primitive binomial generator of $I_D$, where $D$ is a balanced cycle or $D$ comprised of two weighted oriented unbalanced cycles connected by a path of length $\geq 0$, or $D$ comprised of two weighted oriented unbalanced cycles sharing a path.    

\section{Preliminaries} \label{sec2} 

In this section we recall various notions and results require to prove our main results. 
 Let $R=K[x_1,\ldots,x_n]$, where $K$ is a field.  $ n\times m$ matrix of integers $A=[{\bf a_1} \ldots {\bf a_m}]$ with columns ${\bf a_1}, \ldots, {\bf a_m}$. Let $M$ be the monomial ideal in $R$ minimally generated by the monomials $\{\textbf{x}^{\textbf{a}_1},\textbf{x}^{\textbf{a}_2}, \ldots ,\textbf{x}^{\textbf{a}_m}\}$. Let $S=K[e_1, \ldots ,e_m]$ be the polynomial ring in the variables $e_1,\ldots,e_m$. Then define a $K$-algebra homomorphism $\phi$ : $K[e_1, \ldots, e_m]\rightarrow K[x_1, \ldots , x_n]$, as $\phi(e_i) =\textbf{x}^{\textbf{a}_{i}}$. Then the kernel of $\phi$ is called the {\it toric} ideal of $M$ or $A$, and we denote by $I_M$ or $I_{A}$. Then it is known that the irreducible binomials 
 $$\displaystyle \prod_{k=1}^{m} e_k^{p_k} - \prod_{k=1}^{m} e_k^{q_k}, \text{ such that } \displaystyle \sum_{k=1}^m p_k\textbf{a}_k =\sum_{k=1}^m q_k\textbf{a}_k, \text{ for }(p_k,q_k)\neq (0,0),$$ generate $I_M$. Recall that the {\it $k$-minors} of an $m\times n$ matrix $A$ are the determinants of submatrices of $A$ of size $k \times k$.   

 \begin{definition}
For a vector $\mathbf{b} = \left( (-1)^{p_1}b_1,(-1)^{p_2}b_2, \ldots,(-1)^{p_m}b_{l}\right) \in \ZZ^{m}$, with $p_i \geq 1,b_i \geq 0$ integers, define the corresponding binomial in the variables $e_1,\ldots,e_m$ as
$f_{\mathbf{b}} :=f_{\mathbf{b}}^{+} -f_{\mathbf{b}}^{-},$  where 
$$f_{\mathbf{b}}^{+} := \prod\limits_{i=1 \;  (p_i \mbox{ even})}^{m}e_i^{b_i}, \mbox{ and } f_{\mathbf{b}}^{-} := \prod\limits_{i=1 \; (p_i \mbox{ odd})}^{m}e_i^{b_i}.$$  
\end{definition}
Recall that a binomial $f_{\bf b}$ is said to be {\it pure} if $\mbox{gcd}(f_{\bf b}^{+},f_{\bf b}^{-})=1$. For any vector ${\bf b}\in \ZZ^m$, let $[{\bf b}]_{i}$ denote the $i^{th}$ entry of ${\bf b}$. Define $\mbox{supp}({\bf b}) :=\{i:[{\bf b}]_{i}\neq 0\}$ and $\mbox{supp}(f_{\bf b}) :=\{e_{i}:[{\bf b}]_{i}\neq 0\}$. For any $S\subseteq \mbox{supp}({\bf b})$, we denote ${\bf b}\vert_{S}$ the vector with $\mbox{supp}({\bf b}\vert_{S})=S$ and $[{\bf b}\vert_{S}]_i=[{\bf b}]_i$ for all $i \in S$. The {\it sign} of an integer $n$ is defined as $sign(n) := \left\{
	\begin{array}{ll}
		1,  & \mbox{if } n \geq 0, \\
		-1, & \mbox{if } n < 0.
	\end{array}
\right.$ \\ 
 
 A {\it (vertex) weighted oriented graph} is a triplet $D= (V(D),E(D),{\bf w})$, where  $V(D) = \{x_1, \ldots,x_n\}$ is the vertex set of $D$, 
 $$E(D)=\{(x_i,x_j): \text{there is an edge from $x_i$ to $x_j$} \}$$ 
 is the edge set of $D$ and the weight function ${\bf w}:V(D)\rightarrow \NN$. We simply denote the weight function ${\bf w}$ by the vector ${\bf w}=(w_1,\ldots ,w_n)$.  
 The {\it edge ideal} of $D$ is defined as the ideal $I(D)=(x_ix_j^{w_j} : (x_i, x_j)\in E(D))$ in $R$. Then the toric ideal of $D$ is defined as the toric ideal of $I(D)$ and we denote by $I_D$. Thus $I_D=$ker$(\phi)$, where $\phi:K[e:e\in E(D)]\rightarrow R$. Let $E(D)=\{e_1,\ldots,e_m\}$. A {\it leaf} in $D$ is a vertex of degree $1$ in $D$. The {\it outdegree} of a vertex $v$ in a graph $D$ is defined as $|\{e_{}: e=(v,v^{\prime})\mbox{ for some\;} v^{\prime}\in V(D)\}|$. A vertex $v$ is said to be a {\it sink} if its outdegree is zero. 
 Recall that the {\it incidence matrix} of $D$ is an $n \times m$ matrix whose $(i,j)^{th}$ entry $a_{i,j}$ is defined by  
\begin{center}
$a_{i,j}=
\begin{cases}
1,\; \mbox { if } \;e_j = (x_i, x_l) \in E(D)\; \mbox{ for some }\;  1 \le l \le n,  \\
w_i,\; \mbox{ if }\;e_j = (x_l, x_i) \in E(D)\; \mbox{ for some }\; 1 \le l \le n, \\ 0,\; \mbox{ otherwise, } 
\end{cases}  
$
\end{center}
and we denote by $A(D)$. Recall that a weighted oriented even cycle $\mathcal{C}_m$ on $m$ vertices is said to be {\it balanced} if det$(A(\C_m))=0$, that is, $\displaystyle \prod_{k=1}^{m} a_{k,k} =a_{1,m} \prod_{k=1}^{m} a_{k+1,k}$, where $A(\mathcal{C}_m)=[a_{i,j}]_{m \times m}$. We denote $\mbox{Null}(A(D))$, the null space of $A(D)$ over $\QQ$. For a weighted oriented graph $D$, a pure binomial $f_{\bf m}\in I_D$ implies that ${\bf m}\in\mbox{Null}(A(D))$ and $[{\bf m}]_{i}$ denotes the $i$-th entry of ${\bf m}$ corresponding to the edge $e_i\in E(D)$. 

\begin{definition} 
Let $A$ be any $m\times n$ matrix. For any $1\leq k \leq min\{m,n\}$, we denote $M_k(A[i_1,\ldots,i_{m-k}|j_1,\ldots,j_{n-k}])$, the $k^{th}$ minor of $A$ by deleting the rows $i_1,\ldots,i_{m-k}$ and deleting columns $j_1,\ldots,j_{n-k}$ from $A$. If $m > n=k$, then we denote $M_k(A[i_1,\ldots,i_{m-k}|\emptyset])$, the the $k^{th}$ minor of $A$ by deleting the rows $i_1,\ldots,i_{m-k}$ and deleting no column from $A$.  
\end{definition}

\begin{lemma}\cite[Lemma 4.1]{bklo} \label{lem1}  Let $\mathcal{C}_m$ be a weighted oriented $m$-cycle and $f$ is any non-zero element of $I_D$. Then supp($f$) = $E(\mathcal{C}_m)$.  
\end{lemma} 
\begin{theorem}\cite[Theorem 4.3, Algoritham 4.5]{bklo} \label{thm7} 
 If $\mathcal{C}_m$ is weighted oriented cycle, then the toric ideal   $I_{\mathcal{C}_m}$ is non-zero if and only if $\mathcal{C}_m $ is balanced. In fact, in this case, $I_{\C_{m}}$ is a principal ideal. 
\end{theorem}
\noindent 
Below result is a graph theory version of \cite[Proposition 4.13]{s95}. 
\begin{proposition} \label{sec2pro1}
   Let $H$ be a oriented subgraph of a weighted oriented graph $D$ such that $V(D)=V(H)$. Then 
   \begin{enumerate}
       \item[(i)] $I_H=I_D\cap K[e_i: e_i \in E(H)]$, 
       \item[(ii)] $\mathcal{C}_H=\mathcal{C}_D\cap K[e_i: e_i \in E(H)]$,
       \item[(iii)]$\mathcal{U}_H=\mathcal{U}_D\cap K[e_i: e_i \in E(H)]$,
       \item[(iv)] $Gr_H=Gr_D\cap K[e_i: e_i \in E(H)]$. 
   \end{enumerate}
\end{proposition}

\section{Circuit binomials of weighted oriented graphs} \label{sec3}

In this section we describe circuit binomials of toric ideals of weighted oriented graphs. We show that if $D$ is a weighted oriented graph which has at most two unbalanced cycles connected by a path and no two balanced cycles connected by a path in $D$, then the toric ideal $I_D$ is a circuit ideal and $I_D$ is strongly robust. Also, we have shown that the generators of toric ideal of weighted oriented graphs are independent of weights of sink vertices. Note that for any $m\times n$ matrix $A$, its toric ideal $I_{A}$ is generated by $Gr_{A}$.

\begin{lemma} \label{sec3lem1}
    Let $D$ be any weighted oriented graph and $f_{\bf n}\in I_{D}$ be a pure binomial. Then $\mbox{supp}(f_{\bf n})$ can not contain any edge incident with a leaf in $D$.
\end{lemma}
\begin{proof}
    Let $e_{i}$ be the edge incident with leaf $v$ such that $e_{i}\in\mbox{supp}(f_{\bf n})$. Then from $A(D)({\bf n})=0$ and the corresponding to the row with respect to the vertex $v$, we get $[{\bf n}]_{i}x=0$ for some positive integer $x$. This implies that $[{\bf n}]_{i}=0$ i.e. $e_{i}\notin\mbox{supp}(f_{\bf n})$ which is contradiction. This proves the lemma.
\end{proof}

\begin{lemma} \label{sec3lem2}
 Let $D$ be any weighted oriented graph and $f_{\bf m} \neq 0 \in I_D$. Let $v \in V(D)$ of degree $n$. If $(n-1)$ edges of $D$ incident with $v$ are not in supp$(f_{\bf m})$, then the other edge incident with $v$ is not in supp$(f_{\bf m})$. Moreover if the edge $e_i$ incident with $v$ belongs to supp$(f_{\bf m}^{+})( \mbox{ or supp}(f_{\bf m}^{-}))$, then there exists an edge $e_j$ incident with $v$ belongs to supp$(f_{\bf m}^{-})(\mbox{ or supp}(f_{\bf m}^{+}))$.  
\end{lemma}
\begin{proof}
Note that ${\bf m} \in$ Null$(A(D))$. Let $e_1,e_2,\ldots,e_{n}$ be the edges precisely incident with $v$. Suppose $e_1,e_2,\ldots,e_{i-1}, e_{i+1},\ldots,e_{n} \notin$ supp$(f_{\bf m})$. Then $[{\bf m}]_{k}=0$ for $k=1,2,\ldots,i-1,i+1,\ldots,n$. Then from the equation  $A(D){\bf m}={\bf 0}$, we get that $[{\bf m}]_{i}=0$. This implies that $e_i\notin \mbox{supp}(f_{\bf m})$, as required. 

Suppose $e_i\in\mbox{supp}(f_{\bf m}^{+})(\mbox{or}\in \mbox{supp }(f_{\bf m}^{-}))$. Then $[{\bf m}]_{i}>0(\mbox{or}\; [{\bf m}]_{i}<0)$. From the equation $A(D){\bf m}={\bf 0}$, we get $\sum\limits_{k=1}^{i-1}[{\bf m}]_{k}x_k+[{\bf m}]_{i}x_i+\sum\limits_{k=i+1}^{n}[{\bf m}]_{k}x_k=0$ for some positive integers $x_k$'s. This implies that there exists $j$ such that $[{\bf m}]_{j}<0(\mbox{or}\;[{\bf m}]_{j} >0)$. Thus $e_j\in\mbox{supp}(f_{\bf m}^{-})(\mbox{ or supp}(f_{\bf m}^{+}))$. 
\end{proof}

\noindent 
\begin{notation} \label{sec3nota2}
Let $D$ be a weighted oriented graph. 
\begin{enumerate}
    \item For a balanced cycle $\C_i$ in $D$, we know that its toric ideal $I_{\C_i}$ is generated by a single primitive binomial by Theorem \ref{thm7}, say $f_{\bf c_i}$, where ${\bf c_i} \in \mbox{Null}(A(D))$. 
    
    \item For two unbalanced cycles $\C_i, \C_j$ sharing a vertex in $D$, the toric ideal of $\C_i \cup \C_j$ is principal and generated by a primitive binomial by \cite[Theorem 5.1]{bklo}, say $f_{\bf c_i \cup c_j}$, where ${\bf c_i \cup c_j}$ denotes a vector in Null$(A(D))$. 
    
    \item Let ${\C_i}$ and ${\C_j}$ be unbalanced cycles share a path $P$ in $D$. Let $(\C_i \cup \C_j) \setminus P$ denotes the induced subgraph of $D$ whose edge set is $E({\C_i\cup\C_j})\setminus E(P)$. That is, $(\C_i \cup \C_j) \setminus P$ is a cycle and we call it as the outer cycle of $\C_i\cup \C_j$. Then by \cite[Theorem 5.1]{bklo}, the toric ideal of $\C_i \cup \C_j$ is principal, say generated by the primitive binomial $f_{\bf{c_ip c_j}}$, where $\bf{c_ip c_j}$ denotes a vector in Null$(A(D))$. 
    
    \item Let ${\C_i}$ and ${\C_j}$ be two unbalanced cycles connected by a path $P$ in $D$. Then by \cite[Theorem 5.1]{bklo}, the toric ideal of $\C_i\cup P \cup \C_j$ is principal, say generated by the primitive binomial  $f_{\bf{c_i\cup p\cup  c_j}}$, where $\bf{c_i\cup p\cup  c_j}$ denotes a vector in Null$(A(D))$. 
\end{enumerate}   
\end{notation}
\noindent 
Below we describe all circuit binomials in any weighted oriented graph. 

\begin{theorem} \label{sec3pro1}
Let $D$ be a weighted oriented graph. Assume the notation as in \ref{sec3nota2}. 
Then the set of all circuit binomials in $I_D$ is given by 
\begin{eqnarray*}
\mathcal{C}_D &=& \{f_{\bf c}: {\C} \mbox{ is a balanced cycle in } D \} \\ 
& & \cup \{f_{\bf{c_i \cup c_j}} : {\C_i},{\C_j} \mbox{\; are unbalanced cycles  share a vertex in } D \} \\ 
& & \cup \{ f_{\bf{c_ipc_j}} : {\C_i},{\C_j} \mbox{\; are unbalanced cycles sharing a path $P$ in } D \} \\ 
& & \cup \{f_{\bf{c_i\cup p \cup c_j}} : {\C_i},{\C_j} \mbox{\; are unbalanced cycles connected by a path $P$ in } D \}.
\end{eqnarray*}
\end{theorem}
\begin{proof}
 Let $\mathcal{A}$ denotes the right hand side set in the statement. It is easy to see that $\mathcal{A} \subseteq \mathcal{C}_D$ by using Theorem \ref{thm7}, Lemma \ref{lem1}, \cite[Theorem 5.1, Corollary 5.3,]{bklo}. Let $f\in{\C_{D}}$. Using Lemma \ref{sec3lem1}, $\mbox{supp}(f)$ can not contain any edge incident with a leaf. Let $D_{1}$ be a subgraph of $D$ such that $\mbox{supp}(f)=E(D_{1})$. As $0\neq f\in I_{D_{1}}$, then $D_{1}$ is not a unbalanced cycle. Then $D_{1}$ has a subgraph $D_{2}$ such that either $D_{2}$ is a balanced cycle or $D_{2}$ consisting of two unbalanced cycles share a vertex or $D_{2}$ consisting of two unbalanced cycles connected by a path or $D_{2}$ consisting of two unbalanced cycles share a path such that the outer cycle is unbalanced. By Theorem \ref{thm7}, \cite[Theorem 5.1]{bklo}, $I_{D_2}$ is principal, say $I_{D_2}=(g)$. Then $\mbox{supp}(g)\subseteq E(D_2)\subseteq E(D_1)=\mbox{supp}(f)$. Using Proposition \ref{sec2pro1}, $g\in I_{D}$. Since $f\in\mathcal{C}_{D}$, we get $\mbox{supp}(f)=\mbox{supp}(g)\subseteq E(D_{2})$. Since $f\in\mathcal{C}_{D}\subseteq Gr_{D}$, using Proposition \ref{sec2pro1}, $f\in Gr_{D_2}$. Thus $f$ is the generator of $I_{D_2}$. Then $f$ is of the form $f=f_{\bf c}$ or $f=f_{\bf{c_i\cup c_j}}$ or $f=f_{\bf{c_i\cup p\cup c_j}}$ or $f=f_{\bf{c_ipc_j}}$. Hence $f\in\mathcal{A}$ and then ${\C_{D}}\subseteq\mathcal{A}$. Therefore ${\C_{D}}=\mathcal{A}$.        
\end{proof}

\noindent 
Below we give a class of weighted oriented graphs $D$ whose toric ideals $I_D$ are circuit ideals. 

\begin{theorem} \label{sec3thm1}
Let $D$ be a weighted oriented graph and $D$ has at most two unbalanced cycles such that \\ 
(i) if $D$ has exactly two unbalanced cycles connected by a path $P$, then no other balanced cycle in $D$ shares an edge with the path $P$,  \\ 
(ii) no two balanced cycles in $D$ share a path. \\ 
Then 
$Gr_{D}=\mathcal{C}_{D}$. Moreover, $I_{D}$ is strongly robust. 
\end{theorem}
\begin{proof} 
Note that $\C_D \subseteq Gr_D$. 
Let $f_{\bf m}\in Gr_{D}\setminus\mathcal{C}_{D}$. Let $D^{\prime}$ be the subgraph of $D$ such that $E(D^{\prime})=\mbox{supp}(f_{\bf m})$. If $D^{\prime}$ has an  subgraph consisting of two unbalanced cycles ${\C_1}, {\C_2}$ share a path, then as $D$ has at most two unbalanced cycles, the outer cycle of ${\C_1}\cup{\C_2}$, say ${\C_3}$ is balanced and as $f_{\bf m}\notin\mathcal{C}_{D}$, then $E({\C_3})\subsetneqq E(D^{\prime})$. Using Lemma \ref{sec3lem1}, $\mbox{supp}(f_{\bf m})$ can not contain any edge incident with leaf. Since $f_{\bf m}\notin\mathcal{C}_{D}$ and $D$ has no balanced cycles sharing a path, then using Lemma \ref{sec3lem2}, we see that there is a balanced cycle, say ${\C}$ in $D$ such that ${\C}$ shares only one vertex, say $v_1$ with $D^{\prime\prime}$, where $D^{\prime\prime}$ is the subgraph of $D^{\prime}$, with $V(D^{\prime\prime})=V(D^{\prime})\setminus(V({\C})\setminus\{v_1\})$, $E(D^{\prime\prime})=E(D^{\prime})\setminus E({\C})$, $E({\C})\subsetneqq \mbox{supp}(f_{\bf m})=E(D^{\prime})$. Let $V({\C})=\{v_1,v_2,\ldots, v_{n}\}$, $E({\C})=\{e_1,e_2,\ldots, e_{n}\}$, where $e_{i}$ is incident with $v_{i}$ and $v_{i+1}$ for $1\le i\le n-1$, $e_{n}$ is incident with $v_{n}$ and $v_{1}$. Let $V(D)=\{v_1,\ldots,v_n, v_{n+1},\ldots, v_q\}$ and 
$E(D)=\{e_1,\ldots,e_n, e_{n+1},\ldots, e_{q^{\prime}}\}$. 
Let $A(D)=[a_{i,j}]$ be the incidence matrix of $D$ where $v_{i}$ corresponds to $i^{th}$ row and $e_{j}$ corresponds to $j^{th}$ column of $A(D)$. Then $a_{i,j}=0$ for $2\le i\le n$, $j\neq i, j\neq i-1$. Without loss of any generality, assume that $1\in\mbox{supp}({\bf m}_{+})$. Then using repeated applications of Lemma \ref{sec3lem2} , for $2\le i\le n$, $i\in\mbox{supp}({\bf m}_{+})$ for $i$ odd and $i\in\mbox{supp}({\bf m}_{-})$ for $i$ even. Then from $A(D){\bf m}=0$ and the corresponding to each $i^{th}$ row with respect to vertex $v_i$ for $2\le i\le n$, we get $a_{i,(i-1)}[{\bf m_{+}}]_{i-1}=a_{i,i}[{\bf m_{-}}]_{i}$ if $i$ even and $a_{i,(i-1)}[{\bf m_{-}}]_{i-1}=a_{i,i}[{\bf m_{+}}]_{i}$ if $i$ odd. Using above equations, we get
\begin{eqnarray*}
 a_{1,1}[{\bf m_{+}}]_{1} -a_{1,n}[{\bf m_{-}}]_{n} 
 &= &a_{1,1}\frac{a_{2,2}}{a_{2,1}}[{\bf m_{-}}]_{2}-a_{1,n}[{\bf m_{-}}]_{n} \\ 
 &=& a_{1,1}\frac{a_{2,2}}{a_{2,1}}\frac{a_{3,3}}{a_{3,2}}[{\bf m_{+}}]_{3}-a_{1,n}[{\bf m_{-}}]_{n} \\  
 &=&\frac{a_{1,1}a_{2,2}\cdots a_{n,n}}{a_{2,1}a_{3,2}\cdots a_{n,(n-1)}}[{\bf m_{-}}]_{n}- a_{1n}[{\bf m_{-}}]_{n}\\
 && (\text{by repeatedly using the above equations}) \\ 
 &=&\frac{a_{1,1}a_{2,2}\cdots a_{n,n}-a_{1,n}a_{2,1}\cdots a_{n,(n-1)}}{a_{2,1}a_{3,2}\cdots a_{n,(n-1)}}[{\bf m_{-}}]_{n} \\ 
&=&\frac{\mbox{det}(A({\C}))}{a_{2,1}a_{3,2}\cdots a_{n,(n-1)}}[{\bf m_{-}}]_{n} \\&=& 0 ~~~(\text{because det}(A(\mathcal{C}))=0 \text{ as $\mathcal{C}$ is balanced})
\end{eqnarray*}
This implies that ${\bf m}\vert_{\{1,2,\ldots,n\}}\in\mbox{Null}(A({\C}))$ i.e. $f_{\bf m\vert_{\{1,2,\ldots,n\}}}\in I_{\C}=(f_{\bf c})$. Thus $f_{\bf c}^{+}\vert f_{\bf m\vert_{\{1,2,\ldots,n\}}}^{+}$, $f_{\bf c}^{-}\vert f_{\bf m\vert_{\{1,2,\ldots,n\}}}^{-}$ and then $f_{\bf c}^{+}\vert f_{\bf m}^{+}, f_{\bf c}^{-}\vert f_{\bf m}^{-}$ which is contradiction as $f_{\bf m}\in Gr_{D}$ and $f_{\bf m}\neq f_{\bf c}$. Hence $Gr_{D}=\mathcal{C}_{D}$. Since no two balanced cycles in $D$ sharing a path, we have that all elements of $Gr_{D}$  have disjoint supports. This implies that $Gr_D$ is a minimal generating set of $I_D$. Thus $I_{D}$ is strongly robust. 
\end{proof}
\begin{corollary}
  Let $D$ be a weighted oriented graph of one of the below type: 
\begin{enumerate}
    \item[(i)] $D$ consisting of balanced cycles share only a vertex. 
    \item[(ii)] $D$ consisting of balanced cycles ${\C_1}, {\C_2}, \cdots, {\C_n}$ such that ${\C_i}$ and ${\C_{i+1}}$ are connected by a path for $i=1,2,\ldots,n-1$. 
\end{enumerate}
 Then $Gr_{D}=\mathcal{C}_{D}=\{f_{\bf c} : {\C} \mbox{ is cycle in\;}D\}$. 
\end{corollary}
\begin{corollary}
   Let $D$ be a weighted oriented graph and $D$ has at most two unbalanced cycles such that \\ 
(i) if $D$ has exactly two unbalanced cycles connected by a path $P$, then no other balanced cycle in $D$ shares an edge with the path $P$,  \\ 
(ii) no two balanced cycles in $D$ share a path. \\ 
Let $b_D$ denotes the number of balanced cycles in $D$. Suppose $<$ denotes the degree reverse lexicographic term order. Then \\ 
\noindent 
(a) $I_D$ and $in_{<}(I_D)$ are complete intersection ideals. \\ 
\noindent 
(b) $\mu(I_D)=\mu(in_{<}(I_D))= 
    \left\{
	\begin{array}{ll}
		b_D+1,  & \mbox{if $D$ has two unbalanced cycles share a vertex},  \\
           b_D+1,  & \mbox{if $D$ has two unbalanced cycles connected by a path},  \\
		b_D, & \mbox{otherwise}. 
	\end{array}
\right.$ \\
\noindent 
(c) $\beta_{i}(I_D)=\beta_i(in_{<}(I_D))={\mu(I_D) \choose i}$, for all $i$, where $\beta_i(-)$ denotes the $i^{th}$ Betti number. The projective dimension of $I_D$ is equal to $\mu(I_D)$.
\end{corollary}
\begin{proof}
By the Theorem \ref{sec3pro1}, the circuit binomials in $D$ are precisely the primitive binomials corresponding to the balanced cycles in $D$ and the primitive binomial corresponding to the subgraph consists of two unbalanced cycles connected by a path. This implies that $\mu(I_D)=b_D$ or $b_D+1$, accordingly as in statement. By the Theorem \ref{sec3thm1}, we have that $I_D$ is strongly robust, that is, $\C_D=Gr_D$ is a minimal generating set of $I_D$ and in fact, $Gr_D$ is a Gr\"obner basis of $I_D$ with respect to the degree reverse lexicographic order $<$. This implies that $I_D$ and $in_{<}(I_D)$ are complete intersection ideals and $\mu(I_D)=\mu(in_{<}(I_D))$. This proves (a) and (b). The Koszul complexes of $I_D$ and $in_{<}(I_D)$ give minimal free resolutions and $\beta_{i}(I_D)=\beta_i(in_{<}(I_D))={\mu(I_D) \choose i}$, for all $i$, and the projective dimension of $I_D$ is equal to $\mu(I_D)$. This proves (c). 
\end{proof}

\begin{proposition} \label{sec3pro2}
Let $D$ be a weighted oriented graph. Let $D^{\prime}$ be the weighted oriented graph obtained from $D$ by replacing the weights of all sink vertices in $D$ by $1$. Then $I_{D}=I_{D^{\prime}}$. Moreover, the generators of $I_D$ are independent of the weights of the sink vertices in $D$.
\end{proposition}
\begin{proof}
Note that $\mbox{Null}(A(D))=\mbox{Null}(A(D^{\prime}))$ because all the entries in the row corresponding to a sink vertex, are equal. Thus if $f_{\bf m}$ is primitive binomial in $I_{D}$, then it also a primitive binomial in $I_{D^{\prime}}$ and vice-versa. Hence $Gr_{D}=Gr_{D^{\prime}}$. This gives that $I_{D}=I_{D^{\prime}}$.  
\end{proof}
\begin{corollary} \label{sec3cor1}
Let $D$ be weighted oriented graph such that $V^{+}$ are sinks and $G$ be the underlying simple graph of $D$. Then $I_{D}=I_{G}$.    
\end{corollary}
\begin{proof}
 Let $D^{\prime}$ be the graph obtained from $D$ by replacing weights of sinks by 1. Then using Proposition \ref{sec3pro2}, $I_{D}=I_{D^{\prime}}$. Note that $I(D^{\prime})=I(G)$. This implies that $I_{D^{\prime}}=I_{G}$. Hence $I_{D}=I_{G}$.   
\end{proof}

\section{combinatorial formulas for circuit binomials of toric ideals of weighted oriented graphs} \label{sec4} 

In this section, we study combinatorial characterization of primitive binomial generators of toric ideals of weighted oriented graphs. We give explicit combinatorial formulas of generators of $I_D$, where $D$ is any balanced cycle or $D$ is comprised of two unbalanced cycles sharing a path of length $\geq 1$ or $D$ is comprised of two unbalanced cycles connected by a path. 

Let $\C_{n}$ be a weighted oriented $n$-cycle, then we label the vertices and edges as the edge $e_i$ is incident with the vertices $v_i$ and $v_{i+1}$, for $1\leq i \leq n-1$, and the edge $e_n$ is incident with the vertices $v_n$ and $v_1$. We call this labelling as {\it usual labelling} of $\C_n$. The usual labelling of a path $P$ of length $k$ is defined as, the edge $e_i$ is incident with the vertices $v_i$ and $v_{i+1}$, for $1\leq i \leq k$. 

\begin{definition} \label{sec4def1}
(1) For any weighted oriented $n$-cycle $\C_{n}$, let $M_{\ell}(A(\C_n)[i_1,\ldots,i_{n-\ell}|j_1,\ldots,j_{n-\ell}])$ denotes the $\ell^{th}$ minor of $A(\C_n)$ by deleting the rows $i_1,\ldots,i_{n-\ell}$ and deleting the columns $j_1,\ldots,j_{n-\ell}$ from $A(\C_n)$ with respect to the usual labelling of $\C_n$. 
\vskip 0.2cm
\noindent 
(2) For any weighted oriented path $P$ of length $k$, let $M_{\ell}(A(P)[i_1,\ldots,i_{k+1-\ell}|j_1,\ldots,j_{k-\ell}])$ denotes the $\ell^{th}$ minor of $A(P)$ by deleting the rows $i_1,\ldots,i_{k+1-\ell}$ and deleting the columns $j_1,\ldots,j_{k-\ell}$ from $A(P)$ with respect to the usual labelling of $P$.
\end{definition}

\begin{remark} \label{sec4rmk1}
(i) For any matrix $A$ with non-negative integer entries, if $\dim_{\QQ}\mbox{Null}(A)=1$, then $I_A$ is a principal ideal.\\
(ii) For any matrix $A$ with non-negative integer entries, if $\dim_{\QQ}\mbox{Null}(A)=1$ and $f_{\bf a}\in I_{A}$ is a pure binomial, then $f_{\frac{1}{d}{\bf a}}$ is primitive and generates $I_{A}$ where $d=\mbox{gcd}([{\bf a}]_{i} : i\in\mbox{supp}({\bf a}))$.
\begin{proof} (ii)  
Suppose $f_{\frac{1}{d}{\bf a}}$ is not primitive. Then there exists $f_{\bf x}\in I_{A}$ such that $f_{\bf x}^{+}\vert f_{\frac{1}{d}{\bf a}}^{+}, f_{\bf x}^{-}\vert f_{\frac{1}{d}{\bf a}}^{-}$. Since $\dim_{\QQ}\mbox{Null}(A)=1$, we can write ${\bf x}=\lambda\frac{1}{d}{\bf a}$ where $\lambda\in{\QQ}$. Then $\lambda>0$ since $\mbox{supp}({\bf x}_{+})\subseteq\mbox{supp}(\frac{1}{d}{\bf a}_{+})$. Let $\lambda=\frac{p}{q}$ where $\mbox{gcd}(p,q)=1$, $p,q\in{\NN}$. Then $q\vert\frac{[{\bf a}]_{i}}{d}$ for each $i\in\mbox{supp}({\bf a})$. This implies that $q\vert\mbox{gcd}(\frac{[{\bf a}]_{i}}{d})$ for each $i\in\mbox{supp}({\bf a})$ and then $q=1$ as $\mbox{gcd}(\frac{[{\bf a}]_{i}}{d})=1$. Thus ${\bf x}=p\frac{1}{d}{\bf a}$. Since $f_{\bf x}^{+}\vert f_{\frac{1}{d}{\bf a}}^{+}$, then $p$ must be equal to 1. We get ${\bf x}=\frac{1}{d}{\bf a}$ i.e. $f_{\bf x}=f_{\frac{1}{d}{\bf a}}$ which is contradiction. Hence $f_{\frac{1}{d}{\bf a}}$ is primitive.
\end{proof}
\end{remark}

\begin{notation} \label{sec4nota2}
  Let $\C_{2n}$ be a balanced cycle on the vertex set $\{ v_1,\ldots,v_{2n}\}$ with $w(v_i)=w_i$, for $i=1,\ldots,2n$ and edge set $\{e_1,\ldots,e_{2n}\}$, where $e_i$ is incident with $v_i$ and $v_{i+1}$ for $i= 1,2, \ldots, 2n$ under convention that $v_{2n+1}=v_1$. Let $A({\C_{2n}})$ be the incidence matrix of $\C_{2n}$ where $v_{i}$ corresponds to $i^{th}$ row and $e_{i}$ corresponds to $i^{th}$ column for $i=1,2,\ldots,2n$.      
\end{notation}

\noindent 
Below we give an explicit formula for the primitive binomial generator of a balanced cycle in terms of the minors of its incidence matrix. 

\begin{theorem} \label{sec4thm1}
Let $\C_{2n}$ be a balanced cycle as in Notation \ref{sec4nota2}. Then $I_{\C_{2n}}$ is generated by the primitive binomial $f_{\bf{c_{2n}}}$,  where 
$${\bf c_{2n}}=\frac{1}{d}\Bigg ((-1)^{i+1}M_{2n-1}(A(\C_{2n})[1|i]) \bigg)_{i=1}^{2n} \in \ZZ^{2n},$$ 
and $d=\mbox{gcd}(M_{2n-1}(A(\C_{2n})[1|i]))_{i=1}^{2n}$. 
\end{theorem}
\begin{proof}
 Let $A({\C_{2n}})=[a_{i,j}]_{2n \times 2n}$ be the incidence matrix of $\C_{2n}$.   
The incidence matrix $A(\mathcal{C}_{2n})$ is of the following form:  $$\bordermatrix {\text{}&e_1&e_2&e_3&\ldots &e_{2n-1}&e_{2n} \cr  v_1&a_{1,1} &0&0&\ldots &0 &a_{1,2n} \cr v_2&a_{2,1} &a_{2,2}&0 &\ldots &0&0\cr v_3&0 &a_{3,2}&a_{3,3} &\ldots &0&0 \cr v_4&0 &0&a_{4,3} &\ldots &0&0 \cr \vdots&\vdots &\vdots&\vdots &\ldots &\vdots&\vdots \cr v_{2n-1}&0 &0&0&\ldots &a_{2n-1,2n-1}&0\cr v_{2n}&0 &0&0&\ldots &a_{2n,2n-1}&a_{2n,2n}}.$$ 
Let $f_{\bf x}\in I_{\C_{2n}}$ be a pure binomial. Then we have ${\bf x} \in Null(A(\C_{2n}))$. Without loss of generality, assume that $e_{1}\in\mbox{supp}(f_{\bf x}^{+})$. Then using the Lemma \ref{sec3lem2} applied to the vertex $v_2$, we get $e_2\in \mbox{supp}(f_{\bf x}^{-})$. Now repeatedly applying the Lemma \ref{sec3lem2}, we get that for $2\le i\le 2n$, $e_{i}\in\mbox{supp}(f_{\bf x}^{+})$ for $i$ odd and $e_{i}\in\mbox{supp}(f_{\bf x}^{-})$ for $i$ even. Thus ${\bf x}$ is of the form ${\bf x}=((-1)^{i+1}r_{i})_{i=1}^{2n}$, with $r_{i}\in{\NN}$. Then we have $A(\C_{2n}){\bf x}={\bf 0}$. This implies that  
\begin{equation} \label{eq2} 
\begin{aligned}
a_{1,1} r_{1} &=& a_{1,2n} r_{2n}, \text{ and }
 a_{i,i-1}r_{i-1} & =&  a_{i,i}r_{i}, \text{ for }  i=2 \ldots, 2n. 
\end{aligned}\end{equation} 
 Then from equations \eqref{eq2}, we get $r_{i}=\frac{\prod\limits_{k=2}^{i}a_{k,k-1}}{\prod\limits_{k=2}^{i}a_{k,k}}r_{1}$ for $2\le i\le 2n-1$, and $r_{2n}=\frac{\prod\limits_{k=2}^{2n}a_{k,k-1}}{\prod\limits_{k=2}^{2n}a_{k,k}}r_{1}=\frac{M_{2n,2n}}{M_{1,2n}}r_{1}$. Choose $r_{1}=M_{2n-1}(A(\C_{2n})[1|1])$ and substitute in the above expression of $r_i$ we get that $r_{i}=M_{2n-1}(A(\C_{2n})[1|i])$ for $i=2,3,\ldots,2n$. From the first equality of the equation \eqref{eq2} we get  $$a_{1,1}r_{1}=a_{1,1} 
 M_{2n-1}(A(\C_{2n})[1|1])=a_{1,1}\prod\limits_{i=2}^{2n}a_{i,i}=a_{1,2n}\prod\limits_{i=2}^{2n}a_{i,i-1}=a_{1,2n}M_{2n-1}(A(\C_{2n})[1|2n]),$$ 
 where the middle equality holds because $det(A(\C_{2n}))=0$, as ${\C_{2n}}$ is balanced. Thus we have 
 $a_{1,1}r_{1}=a_{1,2n}r_{2n}$. 
 Hence $r_{i}=M_{2n-1}(A(\C_{2n})[1|i])$ satisfies equations \eqref{eq2}. By the \cite[Remark 4.6]{bklo}, we have $\dim_{\QQ}\mbox{Null}(A({\C_{2n}}))=1$ and by using Remark \ref{sec4rmk1}, we get that $f_{\frac{1}{d}{\bf x}}$ is primitive and generates $I_{\C_{2n}}$. 
\end{proof}

\begin{example}
Let $\mathcal{C}_8$ be the weighted oriented graph such that the underlying graph is a cycle having edges $e_1=(v_2, v_1), e_2=(v_3, v_2), e_3=(v_4, v_3), e_4=(v_4, v_5), e_5=(v_5, v_6), e_6=(v_7, v_6), e_7=(v_8, v_7), e_8=(v_8, v_1)$ and weight vector, \textbf{w} = (4, 3, 2, 1, 36, 7, 6, 1). We compute the generator of toric ideal of ${\C_8}$. The incidence matrix of $A({\C_{8}})$ with respect to the usual labelling, is  
$$\begin{bmatrix}
   4 &0 &0 &0&0&0&0&4 \\
  1 &3 &0 &0&0&0&0&0 \\
  0 &1 &2 &0&0&0&0&0 \\
  0 &0 &1 &1&0&0&0&0 \\
  0 &0 &0 &36&1&0&0&0 \\
  0 &0 &0 &0&7&7&0&0 \\ 
  0 &0 &0 &0&0&1&6&0 \\ 
  0 &0 &0 &0&0&0&1&1
 \end{bmatrix}.$$ 
 Then by the notation as in the Theorem \ref{sec4thm1}, we have that $M_{7}(A({\C_8})[1|1])=252$, $M_{7}(A({\C_8})[1|2])=84, M_{7}(A({\C_8})[1|3])=42, M_{7}(A({\C_8})[1|4])=42, M_{7}(A({\C_8})[1|5])=1512, M_{7}(A({\C_8})[1|6])=1512, M_{7}(A({\C_8})[1|7])=252, M_{7}(A({\C_8})[1|8])=252$, and 
 $d=42$. Then by the Theorem \ref{sec4thm1}, we have $I_{\mathcal{C}_8} =\bigg({e_1^{}}^6 e_3^{}{e_5^{}}^{36}{e_7^{}}^6 - {e_2^{}}^2 {e_4^{}}{e_6^{}}^{36}{e_8^{}}^6\bigg)$. Note that Macaulay2 \cite{gs}, gives the same above computed generator of $I_D$.  
\end{example}

\begin{remark} 
If $D$ is the weighted oriented graph comprised of two oriented cycles sharing a vertex such that one cycle is balanced and other is unbalanced. Then 
by \cite[Theorem 5.1, Corollary 5.3]{bklo} we have that $I_D$ is principal and generated by the primitive binomial in $I_D$ corresponding to the balanced cycle in $D$. 
\end{remark}

\begin{notation} \label{sec4nota1}
 Let $D$ be a weighted oriented graph comprised of two cycles $\mathcal{C}_m, \mathcal{C}_n$ sharing a single vertex is labelled as in the below figure:
 \begin{figure}[h!] \centering \includegraphics[scale=0.5]{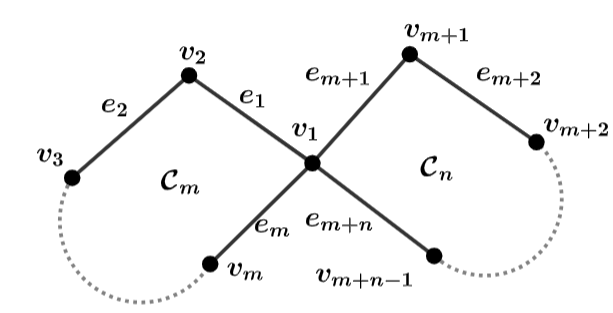}
\end{figure} 

Let $V(\mathcal{C}_m)=\{v_1,v_2, \ldots ,v_m\}$, $V(\mathcal{C}_n)=\{v_{1},v_{m+1}, \ldots ,v_{m+n-1}\}$ and the edge sets $E(\C_m)=\{e_1,e_2, \ldots, e_m\}$, $E(\C_n)=\{e_{m+1},e_{m+2}, \ldots, e_{m+n}\}$, where $e_i$ is incident with $v_i$ and $v_{i+1}$ for $i=1,2,\ldots, m$, $e_{m}$ is incident with $v_{m}$ and $v_{1}$, $e_{m+1}$ is incident with $v_{1}$ and $v_{m+1}$, $e_{m+i}$ is incident with $v_{m+i-1}$ and $v_{m+i}$ for $i=2,3,\ldots,n-1$, $e_{m+n}$ is incident with $v_{m+n}$ and $v_{1}$. Note that $A({\C_m})$, $A({\C_n})$ are submatrices of $A(D)$ with respect to the induced labelling from $D$.      
\end{notation}

\noindent 
Below we give an explicit formula for the primitive binomial generator of a weighted oriented graph comprised of two unbalanced cycles sharing a vertex in terms of the minors of its incidence matrix. 

\begin{theorem} \label{sec4thm2}
Let $D$ be a weighted oriented graph consisting  of two unbalanced cycles  $\mathcal{C}_m, \mathcal{C}_n$ such that these two cycles share only a single vertex as in Notation \ref{sec4nota1}. 
Then the toric ideal $I_D$ is generated by the primitive binomial $f_{\bf{c_{m}\cup{c_{n}}}}$, where ${\bf{c_{m}\cup c_{n}}}$ denotes the vector in $\ZZ^{m+n}$,
$${\bf{c_{m}\cup c_{n}}}=\frac{1}{d}(((-1)^{i+1}pM_{m-1}(A({\C_m})[1|i]))_{i=1}^{m}, ((-1)^{i}qM_{n-1}(A({\C_n})[1|i]))_{i=1}^{n}),$$ 
where $p=\mbox{det}(A({\C_n})), q=\mbox{det}(A({\C_m}))$ and \\ 
$d=\mbox{gcd}((|p|M_{m-1}(A({\C_m})[1|i]))_{i=1}^{m}, (|q|M_{n-1}(A({\C_n})[1|i]))_{i=1}^{n})$.   
\end{theorem}
\begin{proof}
Assume the notation as in the statement. By \cite[Theorem 5.1]{bklo}, we know that $I_{D}$ is principal. Let $f_{\bf x}\in I_{D}$ be a pure binomial. Then ${\bf x}\in Null(A(D))$. Without loss of generality, assume that $e_{1}\in\mbox{supp}(f_{\bf x}^{+})$. Then using the Lemma \ref{sec3lem2} repeatedly, we get that for $2\le i\le m$,  $e_{i}\in\mbox{supp}(f_{\bf x}^{+})$ for $i$ odd and $e_{i}\in\mbox{supp}(f_{\bf x}^{-})$ for $i$ even. Now, apply the Lemma \ref{sec3lem2} at the vertex $v_1$, then we get two possibilities: either $e_{m+1}\in\mbox{supp}(f_{\bf x}^{+})$ or $e_{m+1}\in\mbox{supp}(f_{\bf x}^{-})$. If $e_{m+1}\in\mbox{supp}(f_{\bf x}^{+})$, then by the Lemma \ref{sec3lem2} for $2\le i\le n$, $e_{m+i}\in\mbox{supp}(f_{\bf x}^{+})$ for $i$ odd and $e_{m+i}\in\mbox{supp}(f_{\bf x}^{-})$ for $i$ even. In this case ${\bf x}$ is of the form ${\bf x}={\bf x_{1}}=(((-1)^{i+1}r_{i})_{i=1}^{m},((-1)^{i+1}r_{m+i})_{i=1}^{n})$, for some $r_{i}\in{\NN}$. If $e_{m+1}\in\mbox{supp}(f_{\bf x}^{-})$, then by the Lemma \ref{sec3lem2} for $2\le i\le n$, $e_{m+i}\in\mbox{supp}(f_{\bf x}^{-})$ for $i$ odd and $e_{m+i}\in\mbox{supp}(f_{\bf x}^{+})$ for $i$ even. In this case ${\bf x}$ is of the form ${\bf x}={\bf x_{2}}=(((-1)^{i+1}r_{i})_{i=1}^{m},((-1)^{i}r_{m+i})_{i=1}^{n})$ for some  $r_{i}\in{\NN}$. Therefore in any case we can write ${\bf x}$ is of the form: ${\bf x}=(((-1)^{i+1}r_{i})_{i=1}^{m},((-1)^{i+\alpha}r_{m+i})_{i=1}^{n})$, for some $\alpha \in \NN$.        
Note that $A(D)$ matrix is of the below form \\ 
$$\begin{bmatrix}
   a_{1,1} &0 &\cdots &a_{1,m}&a_{1.m+1}& \cdots & a_{1,m+n} \\
   a_{2,1} & a_{2,2} & \cdots & 0 & 0& \cdots & 0 \\
   \vdots  & \vdots  & \ddots & \vdots &\vdots&\vdots&\vdots  \\
   0 & 0 & \cdots a_{m,m-1} & a_{m,m} &0 &\cdots &0 \\ 0 & 0 & \cdots  & 0 &a_{m+1,m+1} &a_{m+1,m+2} \cdots &0 \\
   \vdots  & \vdots  & \vdots & \vdots &\vdots&\vdots\ddots&\vdots \\ 0&0&\cdots&0&0&\cdots \;\;\;\;\cdots & a_{m+n-1,m+n} 
 \end{bmatrix}$$
Since $A(D){\bf x_{}}={\bf 0}$, we get   
\begin{equation} \label{eq1}  \begin{aligned}
&a_{1,1} r_{1} +(-1)^{m+1}a_{1,m}r_{m} +(-1)^{1+\alpha}a_{1,m+1}r_{m+1}+(-1)^{n+\alpha}a_{1,m+n}r_{m+n} =0\\
&a_{j,j}r_{j}=a_{j,j-1}r_{j-1}, \; \text{ and }\;a_{i,i}r_{i} =a_{i,i+1}r_{i+1},     
\end{aligned}\end{equation} 
for $j=2,3, \ldots, m$,\; and $i=m+1, m+2, \ldots, m+n-1$. Then from above equations, we get that 
$r_m= \frac{\prod\limits_{i=2}^{m}a_{i,i-1}}{\prod\limits_{i=2}^{m}a_{i,i}} r_1$ and $r_{m+n}=\frac{\prod\limits_{i=m+1}^{m+n-1}a_{i,i}}{\prod\limits_{i=m+1}^{m+n-1}a_{i,i+1}} r_{m+1}.$ Now substitute these in the first equation of \eqref{eq1}, we get that \\ 
$a_{1,1}r_{1}  +(-1)^{m+1} a_{1,m} \frac{\prod\limits_{i=2}^{m}a_{i,i-1}}{\prod\limits_{i=2}^{m}a_{i,i}} r_{1} + (-1)^{1+\alpha}a_{1,m+1}r_{m+1} +(-1)^{n+\alpha} a_{1,m+n} \frac{\prod\limits_{i=m+1}^{m+n-1}a_{i,i}}{\prod\limits_{i=m+1}^{m+n-1}a_{i,i+1}} r_{m+1} = 0$. This implies that 
$$\frac{\prod\limits_{k=1}^{m}a_{k,k} +(-1)^{m+1}a_{1,m}\prod\limits_{k=2}^{m}a_{k,k-1}}{\prod\limits_{k=2}^{m}a_{k,k}} r_{1}  = (-1)^{\alpha}\frac{a_{1,m+1}\prod\limits_{k=m+1}^{m+n-1}a_{k,k+1} +(-1)^{n+1}a_{1,m+n}\prod\limits_{k=m+1}^{m+n-1}a_{k,k}}{\prod\limits_{k=m+1}^{m+n-1}a_{k,k+1}}r_{m+1}.$$     
This gives that  $\frac{q}{\prod\limits_{k=2}^{m}a_{k,k}} r_{1} = (-1)^{\alpha} \frac{p}{\prod\limits_{k=m+1}^{m+n-1}a_{k,k+1}} r_{m+1}$. Now compare the signs on both sides we get that $(-1)^{\alpha}=\frac{\mbox{sign}(q)}{\mbox{sign}(p)}$. Therefore we have  $r_{m+1}=\frac{\prod\limits_{i=m+1}^{m+n-1}a_{i,i+1}}{\prod\limits_{i=2}^{m}a_{i,i}}\frac{|q|}{|p|} r_{1}=\frac{M_{n-1}(A({\C_n})[1|1])}{M_{m-1}(A({\C_m})[1|1])}\frac{|q|}{|p|}r_{1}$. From \eqref{eq1}, we get $r_{i}=\frac{\prod\limits_{k=2}^{i}a_{k,k-1}}{\prod\limits_{k=2}^{i}a_{k,k}}r_{1}$, and $r_{m+j}=\frac{\prod\limits_{k=m+1}^{m+j-1}a_{k,k}}{\prod\limits_{k=m+1}^{m+j-1}a_{k,k+1}}r_{m+1}$, for $2\le i\le m, $ and $2\le j\le n$. Choose $r_{1}=M_{m-1}(A({\C_m})[1|1])|p|$. Then we get $r_{i}=M_{m-1}(A({\C_m})[1|i])|p|, $ and $r_{m+j}=M_{n-1}(A({\C_n})[1|j])|q|$, for $2\le i \le m$, and $1\le j \le n$. 
Thus we have ${\bf x}=(((-1)^{i+1}M_{m-1}(A({\C_m})[1|i])|p|)_{i=1}^{m}, ((-1)^{i} \frac{sign(q)}{sign(p)}M_{n-1}(A({\C_n})[1|i])|q|)_{i=1}^{n})$. Then $\frac{1}{d}\mbox{sign}(p){\bf x}\in\mbox{Null}(A(D))$ is the required vector ${\bf c_{m}\cup c_{n}}$. . Now by \cite[Theorem 5.1]{bklo}, we get $\dim_{\QQ}\mbox{Null}(A({D}))=1$. Then by  Remark \ref{sec4rmk1}, the primitive binomial $f_{\frac{1}{d}({\mbox{sign}(p)\bf x_{}})}$ is the generator of $I_{D}$.    
\end{proof}

\begin{notation} \label{sec4nota3}
Let $D$ be a weighted oriented graph comprised of two cycles $\mathcal{C}_m, \mathcal{C}_n$ connected by a path $P$ of length $k$ labelled as shown in the below figure: 

\begin{figure}[h!] \centering \includegraphics[scale=0.5]{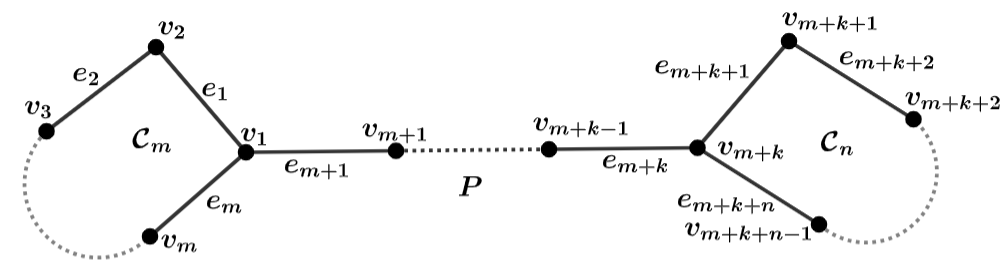}
\end{figure} 

Let $V({\C_m})=\{v_{1},\ldots,v_{m}\}, V(P)=\{v_{1},v_{m+1},\ldots,v_{m+k}\}, V({\C_n})=\{v_{m+k},v_{m+k+1},\ldots,v_{m+k+n-1}\}$, $E({\C_m})=\{e_{1},\ldots,e_{m}\}$, $E(P)=\{e_{m+1},\ldots,e_{m+k}\}$, $E({\C_n})=\{e_{m+k+1},\ldots,e_{m+k+n}\}$, $e_{i}$ is incident with $v_{i}$ and $v_{i+1}$ for $1\le i\le m-1$, $e_{m}$ is incident with $v_{m}$ and $v_{1}$, $e_{m+i}$ is incident with $v_{m+i-1}$ and $v_{m+i}$ for $1\le i\le k+n-1$, $e_{m+k+n}$ is incident with $v_{m+k+n-1}$ and $v_{m+k}$. Note that $A({\C_m}), A({\C_n}), A(P)$ are submatrices of $A(D)$ with respect to the induced labelling from $D$.      
\end{notation}

\noindent 
Below we give an explicit formula for the primitive binomial generator of a weighted oriented graph comprised of two unbalanced cycles connected by a path in terms of the minors of its incidence matrix. 

\begin{theorem} \label{sec4thm3}
Let $D$ be a weighted oriented graph consisting  of two unbalanced cycles  $\mathcal{C}_m, \mathcal{C}_n$ such that these two cycles are connected by a path $P$ of length $k$ as in Notation \ref{sec4nota3}. 
Then the toric ideal $I_D$ is generated by the primitive binomial $f_{\bf{c_{m}\cup p \cup {c_{n}}}}$, where ${\bf{c_{m}\cup p \cup c_{n}}}$ denotes the vector in $\ZZ^{m+k+n}$, 
$${\bf {c_{m}\cup p \cup c_{n}}}=\frac{1}{d}(((-1)^{i+1} p r_{i})_{i=1}^{m},((-1)^{i} pq r_{m+i})_{i=1}^{k},((-1)^{i+k} q r_{m+k+i})_{i=1}^{n}),$$ where
\begin{eqnarray*}
r_{i} &=& M_{k}(A(P)[k+1|\emptyset])M_{m-1}(A({\C_m})[1|i]),\; \text{ for }\;1\le i\le m, \\ 
r_{m+i}&=& M_{k-1}(A(P)[1,k+1|i]),\; \text{ for }\;1\le i\le k, \\ 
r_{m+k+i}&=& M_{k}(A(P)[1|\emptyset])M_{n-1}(A({\C_n})[1|i]),\; \text{ for }\;1\le i\le n,
\end{eqnarray*}
\hspace*{1ex}\;\;\;\;\;$q=\mbox{det}(A({\C_m})),p=\mbox{det}(A({\C_n})), d=\mbox{gcd}((|p|r_{i})_{i=1}^{m}, (|pq|r_{m+i})_{i=1}^k, (|q|r_{m+k+i})_{i=1}^n)$.
\end{theorem}
\begin{proof} 
Assume the notation as in the statement. By \cite[Theorem 5.1]{bklo}, we have that $I_{D}$ is principal and by \cite[Corollary 5.3]{bklo}, we have that support of the generator of $I_{D}$ is equal to $E(D)$. Let $f_{\bf x}\in I_{D}$ be a pure binomial such that $\mbox{supp}(f_{\bf x})=E(D)$. Then ${\bf x} \in Null(A(D))$. Without loss of generality, assume that $e_{1}\in\mbox{supp}(f_{\bf x}^{+})$. Then using the Lemma \ref{sec3lem2} repeatedly, we get that for $2\le i\le m$, $e_{i}\in\mbox{supp}(f_{\bf x}^{+})$ for $i$ odd and $e_{i}\in\mbox{supp}(f_{\bf x}^{-})$ for $i$ even. Now apply the Lemma \ref{sec3lem2} at the vertex $v_1$, then we get two possibilities: either $e_{m+1}\in\mbox{supp}(f_{\bf x}^{+})$ or $e_{m+1}\in\mbox{supp}(f_{\bf x}^{-})$. In either possibility, we apply the Lemma \ref{sec3lem2} repeatedly, we get that $e_{m+i}$ belonging to $\mbox{supp}(f_{\bf x}^{+})$ or $\mbox{supp}(f_{\bf x}^{-})$ for $1\le i\le k$ with $i$ odd or even. Again apply the Lemma \ref{sec3lem2} at the vertex $x_{m+k+1}$, we get two possibilities: either $e_{m+k+1}\in\mbox{supp}(f_{\bf x}^{+})$ or $e_{m+k+1}\in\mbox{supp}(f_{\bf x}^{-})$. In either of the possibilities, by using the Lemma \ref{sec3lem2}, we get that $e_{m+k+i}$ belonging to $\mbox{supp}(f_{\bf x}^{+})$ or $\mbox{supp}(f_{\bf x}^{-})$ for $2\le i\le n$ with $i$ odd or even respectively. Thus the vector ${\bf x}$ is of the form ${\bf x}=(((-1)^{i+1}r_{i})_{i=1}^{m},((-1)^{i+\alpha}r_{m+i})_{i=1}^{k},((-1)^{i+\beta}r_{m+k+i})_{i=1}^{n})$, for some $\alpha,\beta\in{\NN}$ and $r_{i}\in{\NN}$.
Since $A(D){\bf x_{}}={\bf 0}$, we get
\begin{equation} \label{eq3} 
\begin{aligned}
&a_{1,1} r_{1} +(-1)^{m+1}a_{1,m}r_{m} +(-1)^{1+\alpha}a_{1,m+1}r_{m+1} =0,
\end{aligned}\end{equation}
\begin{equation} \label{eq4}  
\begin{aligned}
&a_{i,i-1}r_{i-1}=a_{i,i}r_{i}, \;\;\;a_{j,j}r_{j} =a_{j,j+1}r_{j+1},\; \mbox{for} \;2\le i\le m, m+1\le j\le m+k-1,
\end{aligned}\end{equation}
\begin{equation} \label{eq5}  \begin{aligned}
&(-1)^{k+\alpha}a_{m+k,m+k} r_{m+k} +(-1)^{1+\beta}a_{m+k,m+k+1}r_{m+k+1} +(-1)^{n+\beta}a_{m+k,m+k+n}r_{m+k+n} =0,
\end{aligned}\end{equation}
\begin{equation} \label{eq6}  \begin{aligned}
&a_{l,l}r_{l}=a_{l,l+1}r_{l+1}\; \mbox{for}\;m+k+1\le l\le m+k+n-1.
\end{aligned}\end{equation} 
From equation \eqref{eq4}, we get that $r_m= \frac{\prod\limits_{i=2}^{m}a_{i,i-1}}{\prod\limits_{i=2}^{m}a_{i,i}}r_{1}$ and substituting this in the equation \eqref{eq3}, we get that  $a_{1,1}r_{1}+(-1)^{m+1} a_{1,m}\frac{\prod\limits_{i=2}^{m}a_{i,i-1}}{\prod\limits_{i=2}^{m}a_{i,i}}r_{1}+(-1)^{1+\alpha}a_{1,m+1}r_{m+1}=0$. This implies that 
\begin{equation} \label{eq7}
\frac{\mbox{det}(A({\C_m}))}{\prod\limits_{i=2}^{m}a_{i,i}}r_{1} = (-1)^{\alpha}a_{1,m+1}r_{m+1}. 
\end{equation}
Now compare the sign both sides we get that $(-1)^{\alpha}=\mbox{sign}(\mbox{det}(A({\C_m})))$. 
From the equation \eqref{eq6}, we get $r_{m+k+n}=\frac{\prod\limits_{i=m+k+1}^{m+k+n-1}a_{i,i}}{\prod\limits_{i=m+k+1}^{m+k+n-1}a_{i,i+1}}r_{m+k+1}$ and substitute this in the equation \eqref{eq5}, we get

$$(-1)^{k+\alpha}a_{m+k,m+k}r_{m+k}+(-1)^{1+\beta}a_{m+k,m+k+1}r_{m+k+1} +(-1)^{n+\beta}a_{m+k,m+k+n}\frac{\prod\limits_{i=m+k+1}^{m+k+n-1}a_{i,i}}{\prod\limits_{i=m+k+1}^{m+k+n-1}a_{i,i+1}}r_{m+k+1}=0.$$ 
This implies that 
$$(-1)^{k+\alpha}a_{m+k,m+k}r_{m+k}+(-1)^{1+\beta}\frac{\prod\limits_{i=m+k}^{m+k+n-1}a_{i,i+1}+(-1)^{n+1}a_{m+k,m+k+n}\prod\limits_{i=m+k+1}^{m+k+n-1}a_{i,i}}{\prod\limits_{i=m+k+1}^{m+k+n-1}a_{i,i+1}}r_{m+k+1}=0.$$
This gives that $(-1)^{k+\alpha}a_{m+k,m+k}r_{m+k} = (-1)^{\beta}\frac{\mbox{det}(A({\C_n}))}{\prod\limits_{i=m+k+1}^{m+k+n-1}a_{i,i+1}}r_{m+k+1}$. From the equation \eqref{eq4}, we get  $r_{m+k}=\frac{\prod\limits_{i=m+1}^{m+k-1}a_{i,i}}{\prod\limits_{i=m+1}^{m+k-1}a_{i,i+1}}r_{m+1}$ and substitute this in the above equation, we get that 
$$(-1)^{k+\alpha}a_{m+k,m+k}\frac{\prod\limits_{i=m+1}^{m+k-1}a_{i,i}}{\prod\limits_{i=m+1}^{m+k-1}a_{i,i+1}} r_{m+1} = (-1)^{\beta}\frac{\mbox{det}(A({\C_n}))}{\prod\limits_{i=m+k+1}^{m+k+n-1}a_{i,i+1}}r_{m+k+1}.$$
By the equation \eqref{eq7}, this implies that 
 $$(-1)^{k+\alpha}a_{m+k,m+k}\frac{\prod\limits_{i=m+1}^{m+k-1}a_{i,i}}{\prod\limits_{i=m+1}^{m+k-1}a_{i,i+1}} \frac{\mbox{det}(A({\C_m}))}{\prod\limits_{i=2}^{m}a_{i,i}} r_{1}(-1)^{\alpha} = (-1)^{\beta}\frac{\mbox{det}(A({\C_n}))}{\prod\limits_{i=m+k+1}^{m+k+n-1}a_{i,i+1}}r_{m+k+1}.$$ 
 Now compare the sign both sides we get that $(-1)^{\beta}=(-1)^{k}\frac{\mbox{sign}(\mbox{det}(A({\C_m})))}{\mbox{sign}(\mbox{det}(A({\C_n})))}=(-1)^{k}\frac{\mbox{sign}(q)}{\mbox{sign}(p)}$. 
  Therefore we have $r_{1}=\frac{|p| a_{1,m+1}\prod\limits_{i=m+1}^{m+k-1}a_{i,i+1}\prod\limits_{i=2}^{m}a_{i,i}}{|q|\prod\limits_{i=m+1}^{m+k}a_{i,i}\prod\limits_{i=m+k+1}^{m+k+n-1}a_{i,i+1}}r_{m+k+1}=\frac{|p|M_{k}(A(P)[k+1|\emptyset])M_{m-1}(A({\C_m})[1|1])} {|q|M_{k}(A(P)[1|\emptyset])M_{n-1}(A({\C_n})[1|1])}r_{m+k+1}$. 
 Choose $r_{m+k+1}=|q|M_{k}(A(P)[1|\emptyset])M_{n-1}(A({\C_n})[1|1])$. Then the above equality gives that  $r_{1}=|p|M_{k}(A(P)[k+1|\emptyset])M_{m-1}(A({\C_m})[1|1])$. Since $r_{m+k}=\frac{|p|}{a_{m+k,m+k}\prod\limits_{i=m+k+1}^{m+k+n-1}a_{i,i+1}}r_{m+k+1}$, substitute the value of $r_{m+k+1}$ in it, we get that $r_{m+k}=|p||q|M_{k-1}(A(P)[1,k+1|k])$. Now using the equations \eqref{eq4},\eqref{eq6} and the above computed values of $r_{1}, r_{m+k}, r_{m+k+1}$, one can check that 
 \begin{eqnarray*}
 r_{i} &=& |p| M_{k}(A(P)[k+1|\emptyset])M_{m-1}(A({\C_m})[1|i]),\; \text{ for }\;2\le i\le m, \\ 
r_{m+i} &=& |p||q| M_{k-1}(A(P)[1,k+1|i]),\; \text{ for }\;1\le i\le k-1, \\  
r_{m+k+i} &=& |q| M_{k}(A(P)[1|\emptyset])M_{n-1}(A({\C_n})[1|i]), \; \text{ for }\;2\le i\le n. 
\end{eqnarray*}
Then $\frac{1}{d}\mbox{sign}(p){\bf x}\in\mbox{Null}(A(D))$ is the required vector ${\bf{c_{m}\cup p \cup c_{n}}}$. By \cite[Theorem 5.1]{bklo}, we have $\dim_{\QQ}\mbox{Null}(A({D}))=1$ and by the Remark \ref{sec4rmk1}, we have $f_{\frac{1}{d}({\mbox{sign}(p)\bf x_{}})}$ is a primitive binomial and it generates $I_{D}$. 
\end{proof}

\begin{notation} \label{sec4nota4}
Let $D$ be a weighted oriented graph comprised of two cycles $\C_m, \C_n$ such that these two cycles share a path $P$ of length $k$ labelled as shown in the below figure: 

\begin{figure}[h!] 
\centering \includegraphics[scale=0.5]{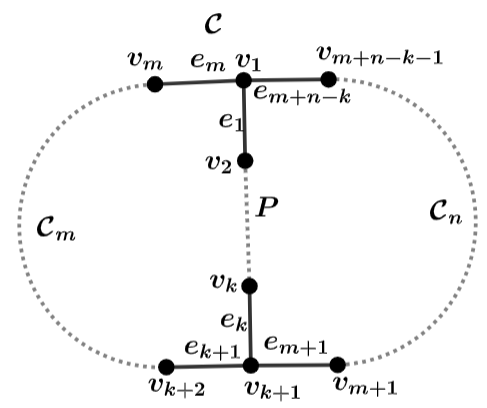}
\end{figure}

Let $V(\mathcal{C}_{m})=\{v_1, v_2, \ldots,v_{k+1}, \ldots  v_{m}\}$,  $V(\mathcal{C}_{n})=\{v_1, v_2, \ldots,v_{k+1},v_{m+1} \ldots,  v_{m+n-k-1}\}$, $V(P)=\{v_1,v_{2}, \ldots, v_{k+1}\}$, $E(\mathcal{C}_{m})=\{e_1, e_2, \ldots,e_{k},e_{k+1} \ldots  e_{m}\}$,  $E(\mathcal{C}_{n})=\{e_1, e_2, \ldots,e_{k},\\e_{m+1} \ldots,  e_{m+n-k}\}$, $E(P)=\{e_1,e_{2}, \ldots, e_{k}\}$. 
Let ${\C}$ be the induced cycle of $D$, whose edge set is given by $E({\C})=(E({\C_m})\cup E({\C_n}))\setminus E(P)$ and $A({\C})$ be the incidence matrix of ${\C}$ with respect to usual labelling where $v_{1}$ corresponds to first row and $v_{m+n-k-1}$ corresponds to last row of $A({\C})$. Note that $A({\C_m}), A({\C_n}), A(P)$ are submatrices of $A(D)$ with respect to the induced labelling from $D$. Note that ${\C_m}\setminus P$, ${\C_n}\setminus P$ denote paths where $E({\C_m}\setminus P)=E({\C_m})\setminus E(P)$, $E({\C_n}\setminus P)=E({\C_n})\setminus E(P)$. We get $V({\C_m}\setminus P)=\{v_{k+1},v_{k+2},\ldots,v_{m},v_{1}\}$, $V({\C_n}\setminus P)=\{v_{k+1},v_{m+1},\ldots,v_{m+n-k-1},v_{1}\}$.   
Let $A({\C_m}\setminus P),\;A({\C_n}\setminus P) $ denote the incidence matrices of ${\C_m}\setminus P,\; {\C_n}\setminus P$ respectively with respect to usual labelling where $v_{k+1}, v_{1}$ corresponds to first row, last row of $A({\C_m}\setminus P)$ and $A({\C_n}\setminus P)$ respectively.      
\end{notation}

\begin{remark}
If $D$ is weighted oriented graph comprised of two unbalanced cycles ${\C_m}, {\C_n}$ share a path as in Notation \ref{sec4nota4} and if the outer cycle ${\C}$ is balanced, then $I_{D}=(f_{\bf c})$, by \cite[Theorem 5.1, Corollary 5.3]{bklo}, where $f_{\bf c}$ is the primitive binomial generator of $I_{\C}$. 
\end{remark}

\noindent 
Below we prove an explicit formula for the primitive binomial generator of a weighted oriented graph comprised of two unbalanced cycles sharing a path of length $\geq 1$ for which the outer cycle is unbalanced, in terms of the minors of its incidence matrix. 

\begin{theorem} \label{sec4thm4}
Let $D$ be a weighted oriented graph comprised of two unbalanced cycles $\C_{m}, \C_{n}$ sharing a path of length $k\geq 1$ as in the Notation \ref{sec4nota4}. Suppose the outer cycle $\C$ is unbalanced. 
Then the toric ideal $I_D$ is generated by the primitive binomial $f_{\bf{c_{m}p{c_{n}}}}$, where ${\bf{c_{m}pc_{n}}}$ denotes the vector in $\ZZ^{m+n-k}$, 
$${\bf {c_{m}pc_{n}}}=\frac{1}{d}\Bigg (((-1)^{i+1}sr_{i})_{i=1}^{k},((-1)^{i+m-k+1}pr_{k+i})_{i=1}^{m-k},((-1)^{i+m-k}qr_{m+i})_{i=1}^{n-k} \Bigg ),$$
\begin{eqnarray*}
\text{ where, } 
r_{i} &=& M_{k-1}(A(P)[1,k+1|i]),\; \text{ for }\;1\le i\le k, \\ 
r_{k+i}&=& M_{m-k-1}(A({\C_m}\setminus P)[1,m-k+1|i]),\; \text{ for }\;1\le i\le m-k, \\ 
r_{m+i}&=& M_{n-k-1}(A({\C_n}\setminus P)[1,n-k+1|i]),\; \text{ for }\;1\le i\le n-k,
\end{eqnarray*}
\hspace*{1ex}$q=\mbox{det}(A({\C_m})),p=\mbox{det}(A({\C_n})),s=\mbox{det}(A({\C})), d=\mbox{gcd}((|s|r_{i})_{i=1}^{k},(|p|r_{k+i})_{i=1}^{m-k}, (|q|r_{m+i})_{i=1}^{n-k})$. 
\end{theorem}
\begin{proof}
We know that $I_{D}$ is principal by \cite[Theorem 5.1]{bklo} and by \cite[Corollary 5.3]{bklo}, we have that the support of the generator of $I_{D}$ is equal to $E(D)$. Let $f_{\bf x}\in I_{D}$ be a pure binomial such that $\mbox{supp}(f_{\bf x})=E(D)$ with ${\bf x} \in Null(A(D))$. Without loss of generality, assume that $e_{1}\in\mbox{supp}(f_{\bf x}^{+})$. Then using the Lemma \ref{sec3lem2} repeatedly, we get that for $2\le i\le k$, $e_{i}\in\mbox{supp}(f_{\bf x}^{+})$ for $i$ odd and $e_{i}\in\mbox{supp}(f_{\bf x}^{-})$ for $i$ even. Now at the vertex $v_{k+1}$, we get possibilities: (either $e_{k+1}\in\mbox{supp}(f_{\bf x}^{+})$ or $e_{k+1}\in\mbox{supp}(f_{\bf x}^{-})$) and (either  $e_{m+1}\in\mbox{supp}(f_{\bf x}^{+})$ or $e_{m+1}\in\mbox{supp}(f_{\bf x}^{-})$). If $e_{k+1}\in\mbox{supp}(f_{\bf x}^{+})$ (resp. $e_{k+1}\in\mbox{supp}(f_{\bf x}^{-})$), then we apply the Lemma \ref{sec3lem2} repeatedly, we get that for $1\le i\le m-k$, $e_{k+i}$ belongs to $\mbox{supp}(f_{\bf x}^{+})$ (resp. $\mbox{supp}(f_{\bf x}^{-})$) if $i$ is odd and $e_{k+i}$ belongs to $\mbox{supp}(f_{\bf x}^{-})$ (resp. $\mbox{supp}(f_{\bf x}^{+})$) if $i$ is even. In the  similar line of arguments, we can have $e_{m+i}$ belonging to $\mbox{supp}(f_{\bf x}^{+})$ or $\mbox{supp}(f_{\bf x}^{-})$, for $1\le i\le n-k$. Thus the vector ${\bf x}$ is of the form ${\bf x}=(((-1)^{i+1}r_{i})_{i=1}^{k},((-1)^{i+\alpha}r_{k+i})_{i=1}^{m-k},((-1)^{i+\beta}r_{m+i})_{i=1}^{n-k})$, for some $\alpha,\beta\in{\NN}$ and $r_{i}\in{\NN}$.
Since $A(D){\bf x_{}}={\bf 0}$, we get
\begin{equation} \label{eq8} 
\begin{aligned}
&a_{1,1} r_{1} +(-1)^{m-k+\alpha}a_{1,m}r_{m} +(-1)^{n-k+\beta}a_{1,m+n-k}r_{m+n-k} =0,
\end{aligned}\end{equation}
\begin{equation} \label{eq9}  
\begin{aligned}
&a_{i,i-1}r_{i-1}=a_{i,i}r_{i}, \; \mbox{for} \;2\le i\le k, 
\end{aligned}\end{equation}
\begin{equation} \label{eq10}  \begin{aligned}
&(-1)^{k+1}a_{k+1,k} r_{k} +(-1)^{1+\alpha}a_{k+1,k+1}r_{k+1} +(-1)^{1+\beta}a_{k+1,m+1}r_{m+1} =0,
\end{aligned}\end{equation}
\begin{equation} \label{eq11}  \begin{aligned}
&a_{i,i}r_{i}=a_{i,i-1}r_{i-1},\;a_{j,j}r_{j}=a_{j,j+1}r_{j+1},\;\mbox{for}\;k+2\le i\le m, \text{and } m+1\le i\le m+n-k-1.
\end{aligned}\end{equation} 
From the equations \eqref{eq9} and \eqref{eq11}, we get that $r_k= \frac{\prod\limits_{i=2}^{k}a_{i,i-1}}{\prod\limits_{i=2}^{k}a_{i,i}}r_{1}$, $r_m= \frac{\prod\limits_{i=k+2}^{m}a_{i,i-1}}{\prod\limits_{i=k+2}^{m}a_{i,i}}r_{k+1}$, and $r_{m+n-k}= \frac{\prod\limits_{i=m+1}^{m+n-k-1}a_{i,i}}{\prod\limits_{i=m+1}^{m+n-k-1}a_{i,i+1}}r_{m+1}$. Substituting these expressions of $r_{m}, r_{m+n-k}$ in the equation \eqref{eq8}, we get that 
\begin{equation*} \begin{aligned}
&a_{1,1}r_{1} +(-1)^{m-k+\alpha} a_{1,m} \frac{\prod\limits_{i=k+2}^{m}a_{i,i-1}}{\prod\limits_{i=k+2}^{m}a_{i,i}}r_{k+1} +(-1)^{n-k+\beta}a_{1,m+n-k}\frac{\prod\limits_{i=m+1}^{m+n-k-1}a_{i,i}}{\prod\limits_{i=m+1}^{m+n-k-1}a_{i,i+1}}r_{m+1}=0, 
\end{aligned}
\end{equation*}
and substitute the above expression of $r_k$ in the equation \eqref{eq10}, we get that 
\begin{equation*} \begin{aligned} 
 &(-1)^{k+1}a_{k+1,k}\frac{\prod\limits_{i=2}^{k}a_{i,i-1}}{\prod\limits_{i=2}^{k}a_{i,i}}r_{1} +(-1)^{1+\alpha}a_{k+1,k+1}r_{k+1} + (-1)^{1+\beta}a_{k+1,m+1}r_{m+1}=0.
 \end{aligned}
\end{equation*}
Solve the last two equations in the unknowns $r_1,r_{k+1}, r_{m+1}$, we get that
$$\frac{r_{1}\prod\limits_{i=k+2}^{m}a_{i,i}\prod\limits_{i=m+1}^{m+n-k-1}a_{i,i+1}}{(-1)^{\alpha+\beta}(-1)^{m-k+1}\mbox{det}(A({\C}))} = \frac{r_{k+1}\prod\limits_{i=2}^{k}a_{i,i}\prod\limits_{i=m+1}^{m+n-k-1}a_{i,i+1}}{(-1)^{\beta}\mbox{det}(A({\C_n}))}=\frac{r_{m+1}\prod\limits_{i=2}^{k}a_{i,i}\prod\limits_{i=k+2}^{m}a_{i,i}}{(-1)^{1+\alpha}\mbox{det}(A({\C_m}))}.$$ 
This implies that
\begin{equation} \label{eq12} \begin{aligned}
\frac{r_{1}}{(-1)^{\alpha+\beta}(-1)^{m-k+1}\prod\limits_{i=2}^{k}a_{i,i}\mbox{det}(A({\C}))} &= \frac{r_{k+1}}{(-1)^{\beta}\prod\limits_{i=k+2}^{m}a_{i,i}\mbox{det}(A({\C_n}))} \\ 
&= \frac{r_{m+1}}{(-1)^{1+\alpha}\prod\limits_{i=m+1}^{m+n-k-1}a_{i,i+1}\mbox{det}(A({\C_m}))}.
\end{aligned} \end{equation} 
Now compare the signs we get that $(-1)^{\alpha}=(-1)^{m-k+1}\frac{\mbox{sign}(\mbox{det}(A({\C})))}{\mbox{sign}(\mbox{det}(A({\C_n})))}$ and $(-1)^{\beta}=(-1)^{m-k}\frac{\mbox{sign}(\mbox{det}(A({\C})))}{\mbox{sign}(\mbox{det}(A({\C_m})))}$. Now choose $r_{1}=\prod\limits_{i=2}^{k}a_{i,i}|\mbox{det}(A({\C}))|=M_{k-1}(A(P)[1,k+1|1])|s|$. Then from the equation \eqref{eq12}, we get that  
\begin{eqnarray*}
r_{k+1} &=&\prod\limits_{i=k+2}^{m}a_{i,i}|\mbox{det}(A({\C_n}))|=M_{m-k-1}(A({\C_m\setminus P})[1,m-k+1|1])|p|,\\  r_{m+1} &=& \prod\limits_{i=m+1}^{m+n-k-1}a_{i,i+1}|\mbox{det}(A({\C_m}))|=M_{n-k-1}(A({\C_n\setminus P})[1,n-k+1|1])|q|.
\end{eqnarray*}
Now substituting these values of $r_1,r_{k+1}, r_{m+1}$ the the equations \eqref{eq9},\eqref{eq11} we get that  
 \begin{eqnarray*}
 r_{i} &=& |s| M_{k-1}(A(P)[1,k+1|i]) ,\; \text{ for }\;2\le i\le k, \\ 
r_{k+i} &=& |p| M_{m-k-1}(A({\C_m\setminus P})[1,m-k+1|i]),\; \text{ for }\;2\le i\le m-k, \\  
r_{m+i} &=& |q| M_{n-k-1}(A({\C_n\setminus P})[1,n-k+1|i])  \; \text{ for }\;2\le i\le n-k. 
\end{eqnarray*}
Now $\frac{1}{d}\mbox{sign}(s){\bf x_{}}\in\mbox{Null}(A(D))$ is the required vector ${\bf{c_{m}pc_{n}}}$. By \cite[Theorem 5.1]{bklo}, we have $\dim_{\QQ}\mbox{Null}(A({D}))=1$ and by the Remark \ref{sec4rmk1}, we have $f_{\frac{1}{d}({\mbox{sign}(s)\bf x_{}})}$ is the primitive binomial and it generates $I_{D}$.
\end{proof}

\begin{example}
 Let $D$ be the weighted oriented graph as shown in the below figure, consisting of two unbalanced cycles ${\C_m}$, ${\C_n}$ share a path $P$ of length $2$ with $m=4, n=4$. Let $\textbf{w} = (1, 2, 3, 4, 5)$ be the weight vector of $D$. 
 
 \begin{figure}[h!] 
\centering \includegraphics[scale=0.5]{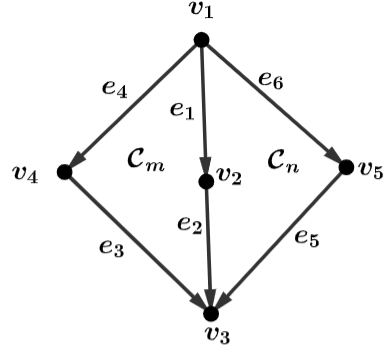}
\end{figure}

Then the incidence matrix of $D$ is  
$$A(D)= \begin{bmatrix}
   1 &0 &0 &1&0&1 \\
  2 &1 &0 &0&0&0 \\
  0 &3 &3 &0&3&0 \\ 
  0 &0 &1 &4&0&0 \\
  0 &0 &0 &0&1&5
 \end{bmatrix}.$$ 
By the notation of the Theorem \ref{sec4thm4}, we have $q=6, p=9, s=3, M_{1}(A(P)[1,3|1])=1, M_{1}(A(P)[1,3|2])=2, M_{1}(A({\C_m}\setminus P)[1,3|1])=4, M_{1}(A({\C_m}\setminus P)[1,3|2])=1, M_{1}(A({\C_n}\setminus P)[1,3|1])=5, M_{1}(A({\C_n}\setminus P)[1,3|2])=1$ and $d=3$. Then by the Theorem \ref{sec4thm4}, we have  ${\bf c_{m}pc_{n}}=\frac{1}{3}(3,-6,36,-9,-30,6)=(1,-2,12,-3,-10,2)$ and $I_{D}=(e_{1}e_{3}^{12}e_{6}^{2}-e_{2}^{2}e_{4}^{3}e_{5}^{10})$. Note that Macaulay2 \cite{gs}, gives the same above computed generator of $I_D$.    
\end{example}

\noindent 
We end the paper with the following question:

\begin{question}
    Can we characterize the primitive binomials of the toric ideal of any weighted oriented graph in terms of the combinatorial information of the graph ? 
\end{question}


\end{document}